\documentclass[1pt]{article}
\usepackage{amsmath, amssymb, amsthm, amsfonts, cases}
\usepackage{mathrsfs}
\usepackage{url}
\usepackage{authblk}
\usepackage[usenames]{color}
\usepackage{geometry}

\newtheorem{rmk}{Remark}[section]
\newtheorem{lemma}{Lemma}[section]
\newtheorem{theorem}{Theorem}[section]

\newtheorem{definition}{Definition}[section]
\newtheorem{pro}{Problem}[section]
\newtheorem{ass}{Assumption}[section]

\renewcommand{\d}{\mathrm{d}}
\newcommand{\eps}{\varepsilon}

\newcommand{\la}{\langle}
\newcommand{\ra}{\rangle}

\allowdisplaybreaks

 \makeatother
\begin{document}

\title{  A Revisit to Optimal Control of Forward-Backward Stochastic Differential System with Observation Noise
 \thanks{This work was supported by the Natural Science Foundation of Zhejiang Province
for Distinguished Young Scholar  (No.LR15A010001),  and the National Natural
Science Foundation of China (No.11471079, 11301177) }}

\date{}

   \author{ Qingxin Meng\thanks{Corresponding author.   E-mail: mqx@zjhu.edu.cn}  \hspace{1cm}
   Qiuhong Shi
 \hspace{1cm}  Maoning Tang
\hspace{1cm}
\\\small{Department of Mathematics, Huzhou University, Zhejiang 313000, China}}

\maketitle
\begin{abstract}

This paper revisits the partial information optimal control problem  considered by Wang, Wu and
Xiong \cite{WWX},
where the  system is
 derived by a controlled forward-backward stochastic differential equation with correlated noises between the system and the observation.
 For this type of partial information optimal control problem, one necessary and one suffcient (a verification theorem) conditions of optimality
are derived using  a unified way.
 We improve
the $L^p-$ bounds on the control  from $L^8$ in \cite{WWX}
to $L^4$ in this paper.
\end{abstract}


\maketitle

\section{ Introduction}
It is well known that  forward-backward stochastic differential
equations (FBSDEs in short) consists of a forward stochastic
differential equation (SDE in short)  of It\^{o} type and a backward
stochastic differential equation (BSDE in short) of Pardoux-Peng
(for details see \cite{PaPe90},\cite{ElPe}).

  FBSDEs are not only encountered in  stochastic optimal control problems
when applying the stochastic maximum principle but also used in
mathematical finance (see Antonelli \cite{Anto}, Duffie and Epstein
\cite{DuEp}, El Karoui, Peng and Quenez \cite{ElPe} for example). It
now becomes more clear that certain important problems in
mathematical economics and mathematical finance, especially in the
optimization problem, can be formulated to be FBSDEs.

 There are two important approaches to the general stochastic optimal
control problem. One is the Bellman dynamic programming principle,
which results in the Hamilton-Jacobi-Bellman equation.  The other
 is the maximum principle.
 Now the maximum principle of forward-backward stochastic systems
 driven by  Brownian motion have been studied extensively in the
 literature. We refer to \cite{ShWz2006, Wu9,Xu95,Yong10}and references therein.

In recent years, there have been growing interests on stochastic
optimal control problems under partial information, partly due to
the applications in mathematical finance. For the partial
information  optimal control problem, the objective is to find an
optimal control for which the controller has less information than
the complete information filtration. In particular, sometimes an
economic model in which there are information gaps among economic
agents can be formulated as a partial information optimal control
problem (see ${\O}$ksendal \cite{okb}, Kohlmann and Xiong \cite{KoX}).

  Recently, Baghery and
${\O}$ksendal \cite{BOK} established a maximum principle of forward systems
with jumps  under partial information.
In 2009,  Meng \cite{Meng} studied
 a partial information stochastic optimal control problem of continuous fully
coupled forward-backward stochastic systems driven by a Brownian
motion. As in \cite{BOK}, the author established  one sufficient (a
verification theorem) and one necessary conditions of optimality.
The main limitation of \cite{BOK} and
 \cite{Meng}
is that the stochastic systems
do not contain observation noise
and the partial information filtration
is too general to have more
practical application.
In 2013, Wang, Wu and Xiong \cite{WWX} studied a partial information optimal control problem derived by forward-backward stochastic systems with correlated noises between the system and the observation. Utilizing a direct method, an approximation method, and a Malliavin derivative method, they established three versions of maximum principle (i.e., necessary condition) for optimal control,
where the following  $L^8-$ bounds
is imposed on the admissible controls:
\begin{eqnarray}\label{eq:111}
  \mathbb E\bigg[\sup_{0\leq t\leq T}|u(t)|^8\bigg].
\end{eqnarray}

The present paper revisits the partial information optimal control problem  considered by Wang, Wu and Xiong \cite{WWX}. Its one of the two main
contributions is that  we  improve
the $L^p-$ bounds on the control  from $L^8-$ bounds \eqref{eq:111}
to the following  $L^4-$ bounds
\begin{eqnarray}
\mathbb E\bigg[\bigg(\int_0^T|u(t)|^2
dt\bigg)^{2}\bigg]<\infty.
\end{eqnarray}
 Another contribution of this paper is
that
we will establish  necessary and
sufficient conditions for an optimal control in a unified way. The
main idea is to get directly a variation formula in terms of the
Hamiltonian and the associated adjoint system which is a linear
forward-backward
stochastic differential equation  and neither the
variational systems nor the corresponding Taylor type expansions of
the state process and the cost functional will be considered.

The paper is organized as follows. In section 2, we formulate the problem and give
various assumptions used throughout the paper. Section 3 is devoted to derive necessary as well as sufficient optimality
conditions in the form of stochastic maximum principles in a unified
way.

\section{Formulation of Problem}

In this section, we introduce
some basic notations  which will be
used in this paper.
Let ${\cal T} : = [0, T]$ denote a finite time index, where $0<T <
\infty$. We consider a complete probability space $( \Omega,
{\mathscr F}, {\mathbb P} )$  equipped with two one-dimensional
standard Brownian motions $\{W(t), t \in {\cal T}\}$ and $\{Y(t),t \in {\cal T}\},$
respectively. Let $%
\{\mathscr{F}^W_t\}_{t\in {\cal T}}$ and $%
\{\mathscr{F}^Y_t\}_{t\in {\cal T}}$ be $\mathbb P$-completed natural
filtration generated by $\{W(t), t\in {\cal T}\}$ and $\{Y(t), t\in {\cal T}\},$ respectively. Set $\{\mathscr{F}_t\}_{t\in {\cal T}}:=\{\mathscr{F}^W_t\}_{t\in {\cal T}}\bigvee
\{\mathscr{F}^Y_t\}_{t\in {\cal T}}, \mathscr F=\mathscr F_T.$   Denote by $\mathbb E[\cdot]$ the expectation
under the probablity $\mathbb P.$
 Let $E$ be a Euclidean space. The inner product in $E$ is denoted by
$\langle\cdot, \cdot\rangle,$ and the norm in $ E$ is denoted by $|\cdot|.$
Let $A^{\top }$ denote the
transpose of the matrix or vector $A.$
For a
function $\psi:\mathbb R^n\longrightarrow \mathbb R,$ denote by
$\psi_x$ its gradient. If $\psi: \mathbb R^n\longrightarrow \mathbb R^k$ (with
$k\geq 2),$ then $\psi_x=(\frac{\partial \phi_i}{\partial x_j})$ is
the corresponding $k\times n$-Jacobian matrix. By $\mathscr{P}$ we
denote the
predictable $\sigma$ field on $\Omega\times [0, T]$ and by $\mathscr %
B(\Lambda)$ the Borel $\sigma$-algebra of any topological space
$\Lambda.$ In the follows, $K$ represents a generic constant, which
can be different from line to line.
Next we introduce some spaces of random variable and stochastic
 processes.
For any $\alpha, \beta\in [1,\infty),$
denote by
$M_{\mathscr{F}}^\beta(0,T;E)$ the space of all $E$-valued and ${%
\mathscr{F}}_t$-adapted processes $f=\{f(t,\omega),\ (t,\omega)\in \cal T
\times\Omega\}$ satisfying
$
\|f\|_{M_{\mathscr{F}}^\beta(0,T;E)}\triangleq{\left (\mathbb E\bigg[\displaystyle%
\int_0^T|f(t)|^ \beta dt\bigg]\right)^{\frac{1}{\beta}}}<\infty, $ by $S_{\mathscr{F}}^\beta (0,T;E)$ the space of all $E$-valued and ${%
\mathscr{F}}_t$-adapted c\`{a}dl\`{a}g processes $f=\{f(t,\omega),\
(t,\omega)\in {\cal T}\times\Omega\}$ satisfying $
\|f\|_{S_{\mathscr{F}}^\beta(0,T;E)}\triangleq{\left (\mathbb E\bigg[\displaystyle\sup_{t\in {\cal T}}|f(t)|^\beta \bigg]\right)^{\frac
{1}{\beta}}}<+\infty,$
by $L^\beta (\Omega,{\mathscr{F}},P;E)$ the space of all
$E$-valued random variables $\xi$ on $(\Omega,{\mathscr{F}},P)$
satisfying $ \|\xi\|_{L^\beta(\Omega,{\mathscr{F}},P;E)}\triangleq
\sqrt{\mathbb E|\xi|^\beta}<\infty,$
by $M_{\mathscr{F}}^\beta(0,T;L^\alpha (0,T; E))$ the space of all $L^\alpha (0,T; E)$-valued and ${%
\mathscr{F}}_t$-adapted processes $f=\{f(t,\omega),\ (t,\omega)\in[0,T]%
\times\Omega\}$ satisfying $
\|f\|_{\alpha,\beta}\triangleq{\left\{\mathbb E\bigg[\left(\displaystyle
\int_0^T|f(t)|^\alpha
dt\right)^{\frac{\beta}{\alpha}}\bigg]\right\}^{\frac{1}{\beta}}}<\infty. $

Consider the following
forward-backward stochastic differential equation

\begin{eqnarray} \label{eq:1}
\left\{
\begin{aligned}
dx(t)=&b(t,x(t),u(t))dt+ \sigma_1(t, x(t), u(t)) dW(t)+\sigma _2(t, x(t), u(t)) dW^u(t),
\\
dy(t)=&f(t,x(t),y(t),z_1(t),z_2(t),u(t))dt+ z_1(t) dW(t)+  z_2(t) dW^u(t),\\
x(0)=&x,\\
y(T)=&\phi(x(T))
\end{aligned}
\right.
\end{eqnarray}
with one observation  processes $Y(\cdot)$ driven by the following stochastic differential equation

\begin{eqnarray}\label{eq:2}
\left\{
\begin{aligned}
dY(t)=&h(t,x(t),u(t))dt+ dW^u(t),\\
Y(t)=& 0,
\end{aligned}
\right.
\end{eqnarray}
where $b: {\cal T} \times \Omega \times {\mathbb R}^n
 \times U  \rightarrow {\mathbb R}^n$,
  $\sigma_1:
{\cal T} \times \Omega \times {\mathbb R}^n
 \times U \rightarrow {\mathbb R}^n$, $\sigma_2: {\cal T} \times \Omega \times {\mathbb R}^n
 \times U \rightarrow {\mathbb R}^n $,
 $f: {\cal T} \times \Omega \times {\mathbb R}^n \times {\mathbb R}^m\times {\mathbb R}^m \times {\mathbb R}^m
 \times U \rightarrow {\mathbb R}^m,
 \phi: \Omega \times {\mathbb R}^n  \rightarrow {\mathbb R^m}$
 and $h: {\cal T} \times \Omega \times {\mathbb R}^n
 \times U \rightarrow {\mathbb R}$  are given random mapping with
 $U$ being a nonempty convex subset of $\mathbb R^k.$ In the above equations, $u(\cdot)$ is our admissible control
 process defined as follows.

\begin{definition}
   An admissible control process is defined as an ${\mathscr F}^Y_t$-adapted process valued in an nonempty
  convex subset $U$  in  $\mathbb R^K$ such that $$\mathbb E\bigg[\bigg(\int_0^T|u(t)|^2
dt\bigg)^{2}\bigg]<\infty.$$
  The set of all admissible controls is denoted by $\cal A.$
\end{definition}

Now we make the following standard  assumptions
 on the coefficients of the equations
 \eqref{eq:1} and
 \eqref{eq:2}.

\begin{ass}\label{ass:1.1}
(i)The coefficients $b$, $\sigma_1,\sigma_2 $ and $h$  are ${\mathscr P} \otimes {\mathscr
B} ({\mathbb R}^n)  \otimes {\mathscr B}
(U) $-measurable. For each $(x,
u) \in \mathbb {R}^n\times U$, $b (\cdot, x,u), \sigma_1(\cdot,x,u)$,
 $\sigma_2 (\cdot, x, u)$
and $h(\cdot, x,u)$ are all $\{\mathscr{F}_t\}_{t\in \cal T}$-adapted processes. For almost all $(t, \omega)\in {\cal T}
\times \Omega$, the mapping
\begin{eqnarray*}
(x,u) \rightarrow \psi(t,\omega,x, u)
\end{eqnarray*}
 is continuous differentiable with respect to $(x,u)$ with
appropriate growths, where $\psi=b, \sigma_1,  \sigma_2$ and $h.$  More precisely, there exists a constant $C
> 0$  such that for  all $x\in \mathbb
R^n, u\in U$ and a.e. $(t, \omega)\in {\cal T} \times \Omega,$
\begin{eqnarray*}
\left\{
\begin{aligned}
& (1+|x|+|u|)^{-1}|\alpha(t,x,u)|
+|\alpha_x(t,x,u)|
+|\alpha_u(t,x,u)|
\leq C, \alpha=b, \sigma_1,
\\&|\beta(t,x,u)| +|\beta_{x}(t,x,u)|
+|\beta_{u}(t,x,u)|\leq
C, \beta=h, \sigma_2,.
\end{aligned}
\right.
\end{eqnarray*}
(ii)The coefficient $f$ is  ${\mathscr P} \otimes {\mathscr
B} ({\mathbb R}^n)\otimes {\mathscr
B} ({\mathbb R}^m)\otimes {\mathscr
B} ({\mathbb R}^m) \otimes {\mathscr
B} ({\mathbb R}^m)  \otimes {\mathscr B}
(U) $-measurable. For each $(x,y,z_1,z_2,
u) \in \mathbb {R}^n\times \mathbb  R^m\times \mathbb R^m\times \mathbb R^m \times U$, $f (\cdot, x,y, z_1, z_2,u)$ is   $\{\mathscr{F}_t\}_{t\in \cal T}$-adapted processes. For almost all $(t, \omega)\in {\cal T}
\times \Omega$, the mapping
\begin{eqnarray*}
(x,y, z_1, z_2,u) \rightarrow f(t,\omega,x,y,z_1,z_2, u)
\end{eqnarray*}
 is continuous differentiable with respect to $(x,y,z_1,z_2,u)$ with
appropriate growths. More precisely, there exists a constant $C
> 0$  such that for  all $(x,y,z_1,z_2,
u) \in \mathbb {R}^n\times \mathbb  R^m\times \mathbb R^m\times \mathbb R^m \times U$ and a.e. $(t, \omega)\in {\cal T} \times \Omega,$
\begin{eqnarray}
\begin{split}
& (1+|x|+|y|+|z_1|+|z_2|+|u|)^{-1}
|f(t,x,y,z_1,z_2,u)|
+|f_x(t,x,y,z_1,z_2,u)|+|f_y(t,x,y,z_1,z_2,u)|
\\&+|f_{z_1}(t,x,y,z_1,z_2,u)|
+|f_{z_2}(t,x,y,z_1,z_2,u)|
+|f_u(t,x,y,z_1,z_2,u)|
\leq C.
\end{split}
\end{eqnarray}
(iii) The coefficient
$\phi$ is ${\mathscr F}_T \otimes {\mathscr B} ({\mathbb
R}^n) $-measurable.  For almost all $(t, \omega)\in [0,T]
\times \Omega$, the mapping

\begin{eqnarray*}
x \rightarrow \phi(\omega,x)
\end{eqnarray*}
 is continuous differentiable with respect to $x$ with
appropriate growths, respectively.
More precisely, there exists a constant $C
> 0$  such that for  all $x\in \mathbb {R}^n$ and a.e. $ \omega\in  \Omega,$

\begin{eqnarray}
\begin{split}
& (1+|x|)^{-1}
|\phi(x)|
+|\phi_x(x)|
\leq C.
\end{split}
\end{eqnarray}

\end{ass}

Now we
begin to discuss  the well- posedness
of \eqref{eq:1} and \eqref{eq:2}.
Indeed, putting \eqref{eq:2} into the state equation \eqref{eq:1}, we get that
\begin{eqnarray} \label{eq:4}
\left\{
\begin{aligned}
dx(t)=&(b-\sigma_2h)(t,x(t),u(t))dt+ \sigma_1(t, x(t), u(t)) dW(t)+\sigma _2(t, x(t), u(t)) dY(t),
\\
dy(t)=&(f(t,x(t),y(t),z_1(t),z_2(t),u(t))
-z_2(t)h(t,x(t),u(t)))dt+ z_1(t) dW(t)+  z_2(t) dY(t),\\
x(0)=&x,\\
y(T)=&\Phi(x(T)).
\end{aligned}
\right.
\end{eqnarray}
Under Assumption \ref{ass:1.1}, for any
admissible control $u(\cdot)\in \cal A,$
we have the following basic
result.

\begin{lemma}\label{lem:3.3}
  Let Assumption \ref{ass:1.1}
  be satisfied. Then for any
  admissible control $u(\cdot)\in
  \cal A,$  the  equation \eqref{eq:4} admits a unique strong
  solution $( x(\cdot),  y(\cdot), z_1(\cdot), z_2(\cdot))\in S_{\mathscr{F}}^4 (0,T;\mathbb R^{n})\times S_{\mathscr{F}}^4 (0,T;\mathbb R^{m})
  \times M_{\mathscr{F}}^2(0,T;L^2 (0,T; \mathbb R^m))\times M_{\mathscr{F}}^2(0,T;L^2 (0,T; \mathbb R^m)).$
Moreover, we have the following
estimate:

\begin{eqnarray}\label{eq:1.8}
\begin{split}
 {\mathbb E} \bigg [ \sup_{t\in \cal T} | x(t) |^4 \bigg
] +{\mathbb E} \bigg [ \sup_{t\in \cal T} | y(t) |^4 \bigg
]+{\mathbb E} \bigg [ \bigg(\int_0^T | z_1(t) |^2 dt\bigg)^{2}
+{\mathbb E} \bigg [ \bigg(\int_0^T | z_2(t) |^2 dt\bigg)^{2}\bigg
]&\leq& K\bigg \{1+|x|^4+\mathbb E\bigg[\Big(\int_0^T |u(t)|^2dt\Big)^{2}\bigg] \bigg\}.
\end{split}
\end{eqnarray}
Further, if $(\bar x(\cdot), \bar y(\cdot), \bar z_1(\cdot),
\bar z_2(\cdot))$ is the unique strong solution
corresponding to another admissible control  $ \bar u(\cdot)\in \cal A,$ then the following
estimate holds:
\begin{eqnarray}\label{eq:1.15}
\begin{split}
&{\mathbb E} \bigg [ \sup_{t\in \cal T} | x (t) - \bar x(t) |^4
\bigg ]+{\mathbb E} \bigg [ \sup_{t\in \cal T} | y (t) - \bar y(t) |^4
\bigg ]+{\mathbb E} \bigg [ \bigg(\int_0^T | z_1(t)-\bar z_1(t) |^2 dt\bigg)^{2}\bigg
]+{\mathbb E} \bigg [ \bigg(\int_0^T | z_2(t)-\bar z_2(t) |^2 dt\bigg)^{2}\bigg
]
\\&\leq  K {\mathbb E} \bigg [ \int_0^T| u(t)- \bar u (t)
|^2 dt\bigg ]^{2}.
\end{split}
\end{eqnarray}
\end{lemma}
\begin{proof}
  The proof can be directly obtained
  by  combining Proposition 2.1 in
  \cite{Mou} and Lemma 2
  in \cite{Xu95}.
\end{proof}

For  the strong solution
 $( x^u(\cdot),  y^u(\cdot), z_1^u(\cdot), z_2^u(\cdot))$ of  the equation
 \eqref{eq:4} associated with
 any given admissible control
$u(\cdot)\in \cal A,$ we  introduce a process
\begin{eqnarray}\label{eq:7}
\begin{split}
  \rho^u(t)
  =\displaystyle
  \exp^{\bigg\{\displaystyle\int_0^t h(s, x^u(s), u(s))dY(s)
  -\frac{1}{2} h^2(s, x^u(s), u(s))ds\bigg\}},
  \end{split}
\end{eqnarray}
which is abviously the solution to the following  SDE
\begin{eqnarray} \label{eq:8}
  \left\{
\begin{aligned}
  d \rho^u(t)=&  \rho^u(t) h(s, x^u(s), u(s))dY(s)\\
  \rho^u(0)=&1.
\end{aligned}
\right.
\end{eqnarray}

 For the stochastic process
 $\rho^u(\cdot),$  we have the
 following  basic result.
 \begin{lemma}\label{lem:3.4}
  Let Assumption \ref{ass:1.1} holds. Then for any $u(\cdot)\in \cal A,$ we have
 for any $\alpha \geq 2,$
\begin{eqnarray}\label{eq:1.9}
\begin{split}
 {\mathbb E} \bigg [ \sup_{{t\in
 \cal T}} | \rho^u(t) |^\alpha \bigg
] \leq K.
\end{split}
\end{eqnarray}
Further, if $ \bar \rho(\cdot)$ is the
process defined by \eqref{eq:7}
or \eqref{eq:8}
corresponding to another
 admissible control $ \bar u(\cdot)\in \cal A,$  then the following
estimate holds
\begin{eqnarray}\label{eq:1.15}
{\mathbb E} \bigg [ \sup_{t\in \cal T} |
\rho^u (t) - \bar \rho (t) |^2
\bigg ]  \leq  K \bigg\{{\mathbb E} \bigg [ \int_0^T| u(t)- \bar u (t)
|^2 dt\bigg ]^{{2}}\bigg\}^{\frac{1}{2}}.
\end{eqnarray}
\end{lemma}

\begin{proof}
  The proof can be directly obtained
  by  combining Proposition 2.1 in
  \cite{Mou}.
\end{proof}

Under Assumption \ref{ass:1.1},
$\rho^u(\cdot)$ is
an  $( \Omega,
{\mathscr F}, \{\mathscr{F}_t\}_{t\in {\cal T}}, {\mathbb P} )-$
martingale. Define a new probability measure $\mathbb P^u$ on $(\Omega, \mathscr F)$ by
\begin{eqnarray}
  d\mathbb P^u=\rho^u(1)d\mathbb P.
\end{eqnarray}
 Then from Girsanov's theorem and \eqref{eq:2}, $(W(\cdot),W^u(\cdot))$ is an
$\mathbb R^2$-valued standard Brownian motion defined in the new probability
space $(\Omega, \mathscr F, \{\mathscr{F}_t\}_{0\leq t\leq T},\mathbb P^u).$
So $(\mathbb P^u, x^u(\cdot), y^u(\cdot),
z^u_1(\cdot), z^u_2(\cdot), \rho^u(\cdot),
 W(\cdot), W^u(\cdot))$ is a weak
solution on $(\Omega, \mathscr F, \{\mathscr{F}_t\}_{t\in \cal
T})$ of  \eqref{eq:1} and
\eqref{eq:2}.

 The cost functional is given by
\begin{eqnarray}\label{eq:13}
  \begin{split}
    J(u(\cdot)=\mathbb E^u\bigg[\int_0^Tl(t,x(t),y(t),z_1(t),z_2(t), u(t))dt+ \Phi(x(T))+\gamma(y(0))\bigg].
  \end{split}
\end{eqnarray}
 where $\mathbb E^u$ denotes the expectation with respect to the
probability space $(\Omega, \mathscr F, \{\mathscr{F}_t\}_{0\leq
t\leq T},\mathbb P^u)$ and  $l:
{\cal T} \times \Omega \times {\mathbb R}^n \times {\mathbb R}^m
\times {\mathbb R}^m\times {\mathbb R}^m
 \times U \rightarrow {\mathbb R},$ $\Phi: \Omega \times {\mathbb R}^n  \rightarrow {\mathbb R}$
 and  $\gamma: \Omega \times {\mathbb R}^m \rightarrow {\mathbb R}$
  are given random mappings
satisfying  the following assumption:

 \begin{ass}\label{ass:1.2}
 $l$ is ${\mathscr P} \otimes {\mathscr
B} ({\mathbb R}^n) \otimes {\mathscr B} ({\mathbb R}^m)\otimes {\mathscr B} ({\mathbb R}^m) \otimes {\mathscr B} ({\mathbb R}^m)  \otimes {\mathscr B}
(U) $-measurable,  $\Phi$ is ${\mathscr F}_T \otimes {\mathscr B} ({\mathbb
R}^n) $-measurable, and $\gamma$ is ${
\mathscr F}_0 \otimes {\mathscr B} ({\mathbb
R}^n) $-measurable. For each $(x,y, z_1,z_2,
u) \in \mathbb {R}^n\times \mathbb R^m
\times \mathbb R^m\times \mathbb R^m\times U$, $f (\cdot, x, y,z_1,z_2, u)$ is  an ${\mathbb F}$-adapted process,  $\Phi(x)$
is an ${\mathscr F}_{T}$-measurable random variable, and  $\gamma(y)$
is an ${\mathscr F}_{0}$-measurable random variable. For almost all $(t, \omega)\in [0,T]
\times \Omega$, the mappings
\begin{eqnarray*}
(x,y,z_1, z_2,u) \rightarrow l(t,\omega,x, y, z_1,z_2, u),
\end{eqnarray*}
\begin{eqnarray*}
x \rightarrow \Phi(\omega,x)
\end{eqnarray*}
and
\begin{eqnarray*}
y\rightarrow \gamma(\omega,y)
\end{eqnarray*}
are continuous differentiable with respect to $(x, y,z_1, z_2,u)$ with
appropriate growths, respectively.
More precisely, there exists a constant $C
> 0$  such that for  all $(x,y, z_1,z_2,
u) \in \mathbb {R}^n\times \mathbb R^m
\times \mathbb R^m\times \mathbb R^m\times U$ and a.e. $(t, \omega)\in [0,T]
\times \Omega,$
\begin{eqnarray*}
\left\{
\begin{aligned}
&
(1+|x|+|y|+|z_1|+|z_2|+|u|)^{-1} (|l_x(t,x,y,z_1,z_2,u)|+|l_y(t,x,y,z_1,z_2,u)|+
|l_{z_1}(t,x,y,z_1,z_2,u)|
\\&\quad\quad+|l_{z_2}(t,x,y,z_1,z_2,u)|
+|l_u(t,x,y,z_1,z_2,u)|)
+(1+|x|^2+|y|^2+|z_1|^2
+|z_2|^2
+|u|^2)^{-1}|l(t,x,y,z_1,z_2,u)|
 \leq C;
\\
& (1+|x|^2)^{-1}|\Phi(x)| +(1+|x|)^{-1}|\Phi_x(x)|\leq
C,
\\
& (1+|y|^2)^{-1}|\gamma(y)| +(1+|y|)^{-1}|\gamma_y(y)|\leq
C.
\end{aligned}
\right.
\end{eqnarray*}
\end{ass}
 Under  Assumption \ref{ass:1.1} and
 \ref{ass:1.2},
 by the estimates \eqref{eq:1.8} and
 \eqref{eq:1.9},  we get that

\begin{eqnarray}
  \begin{split}
    |J(u(\cdot))|\leq & K  \bigg\{\mathbb E\bigg[\sup_{t\in {\cal T}}
    |\rho^u(t)|^2\bigg]\bigg\}^{\frac{1}{2}}
    \bigg\{ \mathbb E\bigg[\sup_{
    {t\in \cal T}}
    |x(t)|^4\bigg]+\mathbb E\bigg[\sup_{
    {t\in \cal T}}
    |y(t)|^4\bigg]
+\mathbb E\bigg[\bigg(\int_0^T|z_1(t)|^2
dt\bigg)^{2}\bigg]
\\&\quad\quad+\mathbb E\bigg[\bigg(\int_0^T|z_2(t)|^2
dt\bigg)^{2}\bigg]+\mathbb E\bigg[\bigg(\int_0^T|u(t)|^2
dt\bigg)^{2}\bigg] +1\bigg\}^{\frac{1}{2}}
    \\  <& \infty,
  \end{split}
\end{eqnarray}
which implies that
the cost functional is well-defined.

Then we  can put forward the following partially observed optimal control problem  in its weak formulation,
 i.e., with changing
 the reference probability space $(\Omega, \mathscr F, \{\mathscr{F}_t\}_{0\leq
t\leq T},\mathbb P^u),$ as follows.

\begin{pro}
\label{pro:1.1} Find an admissible control $\bar{u}(\cdot)\in \cal$ such that
\begin{equation*}  \label{eq:b7}
J(\bar{u}(\cdot))=\displaystyle\inf_{u(\cdot)
\in \cal A}J(u(\cdot)),
\end{equation*}
subject to the
 state equation \eqref{eq:1}, the
 observation equation \eqref{eq:2}
 and the cost functional \eqref{eq:13}.

\end{pro}
Obviously, according to  Bayes' formula,
 the cost functional \eqref{eq:13} can be rewritten as
\begin{equation}\label{eq:15}
\begin{split}
J(u(\cdot ))=& \mathbb E\displaystyle\bigg[%
\int_{0}^{T}\rho^u(t)l(t,x(t),y(t),z_1(t),z_2(t), u(t))dt +\rho^u(T)\Phi(x(T))
+\gamma(y(0))\bigg].
\end{split}
\end{equation}
  Therefore, we  can translate   Problem \ref{pro:1.1}
  into  the following  equivalent optimal control problem in
  its strong formulation, i.e., without changing the  reference
probability space $(\Omega, \mathscr F, \{\mathscr{F}_t\}_{0\leq t\leq T},\mathbb P),$
where $\rho^u(\cdot)$ will be regarded as
an additional  state process besides the
state process $(x^u(\cdot), y^u(\cdot),
z^u_1(\cdot), z^u_2(\cdot)).$

\begin{pro}
\label{pro:1.2} Find an admissible control $\bar{u}(\cdot)$ such that
\begin{equation*}  \label{eq:b7}
J(\bar{u}(\cdot))=\displaystyle\inf_{u(\cdot)
\in \cal A}J(u(\cdot)),
\end{equation*}
 subject to
 the cost functional \eqref{eq:15}
 and  the following
 state equation
\begin{equation}
\displaystyle\left\{
\begin{array}{lll}
dx(t)=&(b-\sigma_2h)(t,x(t),u(t))dt+ \sigma_1(t, x(t), u(t)) dW(t)+\sigma _2(t, x(t), u(t)) dY(t),
\\
dy(t)=&(f(t,x(t),y(t),z_1(t),z_2(t),u(t))
-z_2(t)h(t,x(t),u(t)))dt+ z_1(t) dW(t)+  z_2(t) dY(t),\\
d \rho^u(t)=&  \rho^u(t) h(s, x^u(s), u(s))dY(s),\\
  \rho^u(0)=&1,
\\
x(0)=&x,\\
y(T)=&\Phi(x(T)).
\end{array}%
\right.  \label{eq:3.7}
\end{equation}
\end{pro}
Any $\bar{u}(\cdot)\in \cal A$ satisfying above is called an optimal
control process of Problem \ref{pro:1.2} and the corresponding state
process $(\bar x(\cdot),
\bar y(\cdot),
\bar z_1(\cdot), \bar z_2(\cdot),
\bar \rho(\cdot))$ is called the optimal
state process. Correspondingly $(\bar{u}(\cdot);\bar x(\cdot),
\bar y(\cdot),
\bar z_1(\cdot), \bar z_2(\cdot),
\bar \rho(\cdot))$ is called an optimal pair of Problem \ref{pro:1.2}.
\begin{rmk}
 The present formulation
of the partially observed optimal control problem is quite similar
to a completely observed optimal control problem; the only
difference lies in the admissible class $\cal A$ of controls.
\end{rmk}

\section{Stochastic Maximum Principle}

This section is devoted to establishing
   the stochastic maximum principle of
   Problem \ref{pro:1.1} or   Problem \ref{pro:1.2},
   i.e., establishing the necessary and
   sufficient optimality conditions of Pontryagin's type for
   an admissible control to be optimal. To this end,
for the state equation \eqref{eq:3.7},
we first introduce the
corresponding adjoint equation.

Define the Hamiltonian
 ${\cal H}: \Omega \times {\cal T} \times \mathbb R^n \times \mathbb R^m \times \mathbb R^{m}
  \times \mathbb R^{m}\times U\times
\mathbb R^n\times \mathbb R^n
\times \mathbb R^n
  \times \mathbb R^{m}
  \times \mathbb R\rightarrow \mathbb R$  as follows:
\begin{eqnarray}\label{eq6}
&& {\cal H} (t, x, y, z_1, z_2,u, p,q_1, q_2, k, R_2) \nonumber \\
&& = l (t, x, y, z_1,z_2, u)
+ \langle b (t, x, u),  p \rangle
+ \langle \sigma_1(t, x, u), q_1\rangle + \langle \sigma_2 (t, x, u), q_2\rangle  +
\langle f(t,x,y,z_1,z_2,u), k\rangle  +\langle R_2, h(t,x,u)\rangle \ .
\end{eqnarray}

For any given admissible control pair $(\bar u(\cdot); \bar x(\cdot),
 \bar y(\cdot), \bar z_1(\cdot),
 \bar z_2(\cdot)), $ the corresponding adjoint equation is defined as follows.

\begin{numcases}{}\label{eq:18}
\begin{split}
d\bar r(t)&=- l(t,\bar x(t),
 \bar y(t), \bar z_1(t),
 \bar z_2(t), \bar u(t))dt
 +\bar R_1\left(
t\right)  dW\left(  t\right) +{\bar R}_{2}\left( t\right) dW^{
\bar u}\left( t\right),
\\
d\bar p\left(  t\right)
 &=-{\cal {H}}_{x}\left( t,\bar x(t),
 \bar y(t), \bar z_1(t),
 \bar z_2(t), \bar u(t) \right)dt
 +\bar q_1
\left(  t\right)  dW\left(  t\right)  +{\bar q}_{2}\left( t\right) dW^{
\bar u}\left( t\right),
\\
d\bar k\left(  t\right)  &=-{\cal {H}}_{y}\left(  t,\bar x(t),
 \bar y(t), \bar z_1(t),
 \bar z_2(t), \bar u(t) \right)dt-
{\cal {H}}_{z_1}\left(  t,\bar x(t),
 \bar y(t), \bar z_1(t),
 \bar z_2(t), \bar u(t) \right)  dW\left(  t\right)
  \\&\quad\quad-{\cal {H}}_{z_2}\left(  t,\bar x(t),
 \bar y(t), \bar z_1(t),
 \bar z_2(t), \bar u(t) \right)dW^{
u}\left( t\right),
\\ \bar p(T)&=\Phi_x(\bar x(T))-\phi_x^*(\bar x(T))\bar k(T),
\\ \bar r(T)&=\Phi(\bar x(T)),
\\ \bar k(0)&=-\gamma_y(\bar y(0)).
\end{split}
\end{numcases}
Here we have used the following notation
\begin{eqnarray}\label{eq:19}
\begin{split}
  &{\cal H}_a(t,\bar x(t),
 \bar y(t), \bar z_1(t),
 \bar z_2(t), \bar u(t))
 \\=&{\cal {H}}_{a}\left(  t,\bar x(t),
 \bar y(t), \bar z_1(t),
 \bar z_2(t), \bar u(t), \bar p(t), \bar q_1(t), \bar q_2(t),
  \bar k(t), \bar R_2(t)-\bar\sigma_2^*(t, x(t), u(t))\bar p(t)-\bar z_2 ^*(t)
  \bar k(t) \right)
  \end{split}
\end{eqnarray}
where $a=x, y,z_1, z_2,u.$

Note the adjoint equation \eqref{eq:18}
is a forward-backward
stochastic differential equation whose solution
consists of  an 7-tuple process $(\bar p(\cdot),\bar q_1(\cdot), \bar { q}_2(\cdot),\bar k(\cdot),\bar r(\cdot),\bar R_1(\cdot),\bar R_2(\cdot ) ).$
Under Assumptions \ref{ass:1.1} and
\ref{ass:1.2},  by Proposition 2.1 in
  \cite{Mou} and Lemma 2
  in \cite{Xu95} ,
it is easily to see that  the adjoint equation \eqref{eq:18} admits
a unique solution $(\bar p(\cdot),\bar q_1(\cdot), \bar { q}_2(\cdot),\bar k(\cdot),\bar r(\cdot),\bar R_1(\cdot),\bar R_2(\cdot ) )\in S_{\mathscr{F}}^4(0,T;
\mathbb R^n)\times
M_{\mathscr{F}}^2(0,T;L^2 (0,T; \mathbb R^n)) \times M_{\mathscr{F}}^2(0,T;L^2 (0,T; \mathbb R^n))\times S_{\mathscr{F}}^4(0,T;
\mathbb R^m)\times S_{\mathscr{F}}^4(0,T;
\mathbb R)\times M_{\mathscr{F}}^2(0,T;L^2 (0,T; \mathbb R))\times M_{\mathscr{F}}^2(0,T;L^2 (0,T; \mathbb R^n)),
$  also called the adjoint process corresponding
the admissible pair $(\bar{u}(\cdot);\bar x(\cdot),
\bar y(\cdot),
\bar z_1(\cdot), \bar z_2(\cdot),
\bar \rho(\cdot))$.

Let $({u}(\cdot);x^u(\cdot),
 y^u(\cdot),
 z^u_1(\cdot),  z^u_2(\cdot),
\rho^u(\cdot))$ and $(\bar{u}(\cdot);\bar x(\cdot),
\bar y(\cdot),
\bar z_1(\cdot), \bar z_2(\cdot),
\bar \rho(\cdot))$ be
two admissible pairs. Next we represent the difference $J (u (\cdot)) - J (\bar u (\cdot))$
in terms of the Hamiltonian ${\cal H}$ and the adjoint process $(\bar p(\cdot),\bar q_1(\cdot), \bar { q}_2(\cdot),\bar k(\cdot),\bar r(\cdot),\bar R_1(\cdot),\bar R_2(\cdot ) )$
as well as other relevant expressions.

To unburden our notation, we will use the
following abbreviations:
\begin{eqnarray}\label{eq:3.16}
\left\{
\begin{aligned}
& \gamma^u (0)
= \gamma( y^u (0)) \ ,
 \bar \gamma(0) =
 \gamma( \bar y (0)), \
 \\& \phi^u (T)
= \phi( x^u (T)),
 \bar \phi(T) =  \phi( \bar x (T)) \ ,
\\& \Phi^u (T)
= \Phi( x^u (T)),
 \bar \Phi(T) =  \Phi( \bar x (T)),\\
& \alpha^u (t) = \alpha ( t, x^u (t),u (t) ),\\
& \bar \alpha (t) = \alpha( t, \bar x (t), \bar
u(t) ), \quad \alpha =  b, \sigma_1,
\sigma_2, h,
\\
& \beta^u (t) = \alpha ( t, x^u (t),
y^u(t), z^u_1(t),z^u_2(t),
u (t) ),\\
& \bar \beta (t) = \alpha( t, \bar x (t), \bar y(t), \bar z_1(t),\bar z_2(t),\bar
u(t) ), \quad \beta =  f, l.
\end{aligned}
\right.
\end{eqnarray}

\begin{lemma}\label{lem4}
Let Assumptions \ref{ass:1.1} and \ref{ass:1.2} be satisfied.
Using the notations \eqref{eq:19} and
\eqref{eq:3.16},  we have
\begin{eqnarray}\label{eq:21}
&&J (u (\cdot)) - J (\bar u(\cdot))\nonumber
\\ &=& {\mathbb E} ^{\bar u}
\bigg[\int_0^T \bigg [ { \cal H}(t, x^u(t),
y^u(t), z^u_1(t), z^u_2(t), u(t)) - {\cal H} (t, \bar x(t),
\bar y(t), \bar z_1(t),\bar z_2(t),  \bar u(t)) \nonumber \\&& - \big < {\cal H}_x (t, \bar x(t),
\bar y(t), \bar z_1(t), \bar z_2(t),  \bar u(t)),x^u (t) - \bar
x (t)
\big>\nonumber
\\&&- \big <{ \cal H}_y (t, \bar x(t),
\bar y(t), \bar z_1(t), \bar z_2(t),  \bar u(t)),
y^u (t) - \bar
y (t)\big>\nonumber
\\&&- \big <{\cal H}_{z_1} (t, \bar x(t),
\bar y(t), \bar z_1(t),\bar z_2(t),  \bar u(t)),
 z^u_1 (t) - \bar
z_1 (t)\big>\nonumber
\\&&- \big <{\cal H}_{z_2} (t, \bar x(t),
\bar y(t), \bar z_1(t),\bar z_2(t),  \bar u(t)),
 z^u_2 (t) - \bar
z_2 (t)\big>\nonumber
\\&&-
 \langle (\sigma_2^u(t)-\bar{\sigma}_2(t))
 (h^u(t)-\bar
 h(t)), \bar p(t)\rangle \nonumber
 \\&& -
 \langle (z_2^u(t)-\bar{z}_2(t))
 (h^u(t)-\bar
 h(t)), \bar k(t)\rangle dt\bigg]
\nonumber \\
&& + {\mathbb E}^{\bar u} \big [ \Phi^u(T)
 - \bar \Phi (T) - \left <
x^u (T) - \bar x (T), \bar \Phi_x (T)  \right
> \big ]
\nonumber \\
&& - {\mathbb E}^{\bar u}
\big [ \la \phi^u (T) - \bar \phi (T), \bar k(T)\ra - \left <
\bar \phi_x^*(T)\bar k(T),
x^u (T) - \bar x (T) \right
> \big ]
\nonumber \\
&& + {\mathbb E}
 \big [ \gamma^u(0) - \bar \gamma (0) - \left <
y^u (0)
- \bar y (0), \bar \gamma_y (0)  \right
> \big ]
\nonumber\nonumber
\\&&+\mathbb E\bigg[\int_0^T \bar R_2(t)(\rho^u(t)-\rho^{\bar u}(t))(h^u(t)-\bar
 h(t))dt\bigg] \nonumber
\\&&+\mathbb E\bigg[\int_0^T(l^u(t)-\bar l(t))( \rho^u(t)-\bar
\rho(t))dt\bigg]\nonumber
\\&&+\mathbb E \bigg[ ( \rho^u (T) - \bar \rho(T)) (\Phi^u (T) -
\bar \Phi (T))\bigg].
\end{eqnarray}
\end{lemma}

\begin{proof}
  From \eqref{eq:2}, it is easy to check that the $(x^u(\cdot), y^u(\cdot), z^u(\cdot))$ satisfies the following FBSDE:

\begin{eqnarray}
\left\{
\begin{aligned}
dx(t)=&\big[b^{u}(t)+\sigma _2^u(t)(\bar  h(t)-h^u(t))\big]dt+
\sigma^u_1(t) dW(t)+\sigma _2^u(t) dW^{\bar u}(t)
\\
dy(t)=&\big[f^u(t)+z _2(t)( \bar h(t)- h^u(t))\big]dt+ z_1(t) dW(t)+  z_2(t) dW^{\bar u}(t)\\
x(0)=&x,\\
y(T)=&\phi(x(T))
\end{aligned}
\right.
\end{eqnarray}
Therefore $(x^u(t)-\bar x(t),
y^u(t)-\bar y(t), z^u(t)-\bar z(t))$ satisfies the following FBSDE:

\begin{eqnarray}
\left\{
\begin{aligned}
dx(t)-\bar x(t)=&\big[b^u(t)-\bar b(t)+\sigma _2^u(t)( h^{\bar u}(t)-h^u(t))\big]dt+ [\sigma^u_1(t)-\bar\sigma_1(t)] dW(t)+[\sigma _2^u(t)-\bar\sigma _2(t)]dW^{\bar u}(t)
\\
dy(t)-\bar y(t)=&\big[f^u(t)-\bar f(t) +z _2(t)( \bar h(t)-h(t))\big]dt+ [z_1(t)-\bar z_1(t)] dW(t)+  [z_2(t)-\bar z_2(t)] dW^{\bar u}(t)\\
x(0)-\bar x(0)=&0,\\
y(T)-\bar y(T)=&\phi(x(T))-\phi(\bar x(T)).
\end{aligned}
\right.
\end{eqnarray}

From \eqref{eq:18}, we know that
$(\bar p(\cdot),\bar q_1(\cdot),\bar q_2(\cdot),\bar k(\cdot) )$ satisfies the
following   FBSDE

\begin{numcases}{}\label{eq:3.3}
\begin{split}
d\bar p\left(  t\right)
 &=-{\cal {H}}_{x}\left( t,\bar x(t),
 \bar y(t), \bar z_1(t),
 \bar z_2(t), \bar u(t) \right)dt
 +\bar q_1
\left(  t\right)  dW\left(  t\right)  +{\bar q}_{2}\left( t\right) dW^{
\bar u}\left( t\right),
\\
d\bar k\left(  t\right)  &=-{\cal {H}}_{y}\left(  t,\bar x(t),
 \bar y(t), \bar z_1(t),
 \bar z_2(t), \bar u(t) \right)dt-
{\cal {H}}_{z_1}\left(  t,\bar x(t),
 \bar y(t), \bar z_1(t),
 \bar z_2(t), \bar u(t) \right)  dW\left(  t\right)
  \\&\quad\quad-{\cal {H}}_{z_2}\left(  t,\bar x(t),
 \bar y(t), \bar z_1(t),
 \bar z_2(t), \bar u(t) \right)dW^{
\bar u}\left( t\right),
\\ \bar p(T)&=\bar\Phi_x(T)-\bar \phi_x^*(T)\bar k(T),
\\ \bar k(0)&=-\bar \gamma_y(0),
\end{split}
\end{numcases}

It is easy to check that $(\bar r(\cdot),
\bar R_1(\cdot),
\bar R_2(\cdot))$  satisfying the following BSDE

\begin{numcases}{}\label{eq:3.3}
\begin{split}
d\bar r(t)&=-[\bar l(t)+\bar R_2(t)\bar h(t)]dt+\bar R_1\left(
t\right) dW\left(  t\right) +{\bar R}_{2}\left( t\right) dY(t),
\\ \bar r(T)&=\bar\Phi(T),
\end{split}
\end{numcases}

From  the definition of the cost function $J(u(\cdot)),$
we have

\begin{eqnarray}\label{eq:26}
\begin{split}
 J(u(\cdot))-J(\bar u(\cdot))=&
 \mathbb E^u\bigg[\int_0^Tl^u(t)dt+ \Phi^u(T)+\gamma^u(0)\bigg]-\mathbb E^{\bar u}\bigg[\int_0^T\bar l(t)dt+ \bar\Phi(T)+\bar\gamma(0)\bigg]
 \\=& \mathbb E\bigg[\int_0^T (\rho^u(t)l^u(t)-\bar\rho(t)\bar l(t))dt\bigg]+ \mathbb E[\rho^u(T)\Phi^u(T)-\bar \rho (t)\bar \Phi(T)]
 +\mathbb E[\gamma^u(0)-\bar \gamma(0)]
 \\=& \mathbb E^{\bar u}
 \bigg[\int_0^T [l^u(t)-\bar l(t)]dt
 \bigg]+ \mathbb E^{\bar u}[\Phi^u(T)-\bar \Phi(T)]+ \mathbb E\bigg[\int_0^T (\rho^u(t)-\bar\rho(t)) l^u(t)dt\bigg]
\\&+ \mathbb E[(\rho^u(T)-\bar \rho (T))\Phi^u(T)]
 +\mathbb E[\gamma^u(0)-\bar \gamma(0)]
\end{split}
\end{eqnarray}

From the definition of $\cal H,$ we get that

\begin{eqnarray}\label{eq:27}
  \begin{split}
   \mathbb E^{\bar u}\bigg[\int_0^T (l^u(t)-\bar l(t))dt\bigg]=&{\mathbb E} ^{\bar u}
\bigg[\int_0^T \bigg ( { \cal H}(t, x^u(t),
y^u(t), z^u_1(t), z^u_2(t), u(t)) - {\cal H} (t, \bar x(t),
\bar y(t), \bar z_1(t),\bar z_2(t),  \bar u(t))\bigg)dt\bigg]
\\&- \mathbb E^{\bar u}\bigg[\int_0^T \bigg (\langle \bar p(t), b^u (t) - \bar b (t)\rangle + \langle \bar q_1(t), \sigma_1^u(t) - \bar \sigma_1 (t)\rangle
+\langle \bar  q_2(t), \sigma _2^u(t) - {\bar \sigma}_2 (t)
\rangle
\\&+ \langle \bar k(t),  f^u(t) - \bar f (t)\rangle+\langle \bar R_2(t)-\bar \sigma_2^*(t)\bar p(t)-\bar z_2 ^*(t)\bar k(t),h^u(t)-\bar h(t)\rangle \bigg)dt\bigg]
  \end{split}
\end{eqnarray}

Applying It\^{o} formula to $\la \bar p(t), x^u(t)-\bar x(t)\ra
+\la \bar k(t), y^u(t)-\bar y(t)\ra$ and taking expectation under $P^{\bar u}$, we have

\begin{eqnarray}
   && \mathbb E^{\bar u} \big[ \la \bar \Phi_x(T)-\bar\phi_x^*(T)\bar k(T), x^u(T)-\bar x(T)\ra\big]+\mathbb E^{\bar u} \big[ \la \bar k(T), \phi^u(T)-\bar \phi(T)\ra\big]\nonumber
\\=&&\mathbb E^{\bar u}\bigg[\int_0^T \bigg (\langle \bar p(t), b^u (t) - \bar b (t)\rangle + \langle \bar q_1(t), \sigma_1^u (t) - \bar \sigma_1 (t)\rangle
+ \langle \bar q_2(t), \sigma _2^u(t) - {\bar \sigma}_2 (t)\rangle
\nonumber
\\&&+ \langle \bar k(t), f^u(t) - \bar f (t)\rangle
+\langle \bar p(t), \sigma_2^u(t)(\bar h(t)- h^u(t))\rangle +\langle \bar k(t), z_2^u(t)(\bar h(t)- h^u(t))\rangle \bigg)dt
\bigg]
\nonumber
\\&&-\mathbb E^{\bar u} \bigg[\int_0^T \la { \cal H}_x (t, \bar x(t),
\bar y(t), \bar z(t),
 \bar u(t)), x^u (t) - \bar
x (t)\ra dt\bigg]
\nonumber
\\&&-\mathbb E^{\bar u} \bigg[\int_0^T \la { \cal H}_y (t, \bar x(t),
\bar y(t), \bar z(t),
 \bar u(t)), y^u (t) - \bar
y (t)\ra dt\bigg]
\nonumber
\\&&-\mathbb E^{\bar u} \bigg[\int_0^T \la { \cal H}_{z_1}(t, \bar x(t),
\bar y(t), \bar z(t),
 \bar u(t)), z_1^u (t) - \bar
z_1 (t)\ra dt\bigg]
\nonumber
\\&&-\mathbb E^{\bar u} \bigg[\int_0^T \la { \cal H}_{z_2}(t, \bar x(t),
\bar y(t), \bar z(t),
 \bar u(t)), z_2^u (t) - \bar
z_2 (t)\ra dt\bigg]
\nonumber
\\&&-{\mathbb E}
 \big [  \left <
y^u (0)
- \bar y (0), \bar \gamma_y (0)  \right
> \big ]
\bigg]
\end{eqnarray}

which implies that

\begin{eqnarray} \label{eq:29}
   &&\mathbb E^{\bar u}\bigg[\int_0^T \bigg (\langle \bar p(t), b^u (t) - \bar b (t)\rangle + \langle \bar q_1(t), \sigma_1^u (t) - \bar \sigma_1 (t)\rangle
+ \langle \bar q_2(t), \sigma _2^u(t) - {\bar \sigma}_2 (t)\rangle+ \langle \bar k(t), f^u(t) - \bar f (t)\rangle \bigg)dt\bigg]\nonumber
\\=&& \mathbb E^{\bar u} \big[ \la \bar \Phi_x( T), x^u(T)-\bar x(T)\ra\big]+\mathbb E^{\bar u} \big[ \la \bar k(T), \phi^u(T)-\bar \phi(T)-\bar\phi_x(T)(x^u(T)-\bar x(T))\ra\big]\nonumber
\\&&+ \mathbb E\bigg[\int_0^T\bigg(\langle \bar p(t), \sigma_2^u(t)(h^u(t)-\bar h(t))\rangle +\langle \bar k(t), z_2^u(t)(h^u(t)-\bar h(t))\rangle \bigg)dt\bigg]\nonumber
\\&&+\mathbb E^{\bar u} \bigg[\int_0^T \la { \cal H}_x (t, \bar x(t),
\bar y(t), \bar z(t),
 \bar u(t)), x^u (t) - \bar
x (t)\ra dt\bigg]\nonumber
\\&&+\mathbb E^{\bar u} \bigg[\int_0^T \la { \cal H}_y (t, \bar x(t),
\bar y(t), \bar z(t),
 \bar u(t)), y^u (t) - \bar
y (t)\ra dt\bigg]\nonumber
\\&&+\mathbb E^{\bar u} \bigg[\int_0^T \la { \cal H}_{z_1}(t, \bar x(t),
\bar y(t), \bar z(t),
 \bar u(t)), z_1^u (t) - \bar
z_1 (t)\ra dt\bigg] \nonumber
\\&&+\mathbb E^{\bar u} \bigg[\int_0^T \la { \cal H}_{z_2}(t, \bar x(t),
\bar y(t), \bar z(t),
 \bar u(t)), z_2^u (t) - \bar
z_2 (t)\ra dt\bigg]\nonumber
\\&&+{\mathbb E}
 \big [  \left <
y^u (0)
- \bar y (0), \bar \gamma_y (0)  \right
> \big ]
\bigg]
\end{eqnarray}

Applying It\^{o} formula to
$(\rho^u(t)-\bar \rho(t)) \bar r(t), $
we have
\begin{eqnarray}
  \begin{split}
    \mathbb E [ (\rho^u(T)-\bar\rho(T))\bar \Phi(T)]= &-\mathbb E\bigg[\int_0^T(\rho^u(t)-\bar\rho(t))(\bar l(t)+\bar R_2(t)\bar h(t))dt\bigg]
    +\mathbb E\bigg[\int_0^T\bar R_2(t)(\rho^u(t)h^u(t)-\bar\rho(t)\bar  h(t))dt\bigg],
  \end{split}
\end{eqnarray}

which implies that

\begin{eqnarray}\label{eq:31}
  \begin{split}
    \mathbb E  [(\rho^u(T)-\bar\rho(T))\bar \Phi(T)]+\mathbb E\bigg[\int_0^T(\rho^u(t)
    -\bar\rho(t))\bar l(t)dt\bigg]=
    \mathbb E\bigg[\int_0^T\bar R_2(t)\rho^u(t)(h^u(t)-\bar  h(t))dt\bigg]
  \end{split}
\end{eqnarray}

Putting \eqref{eq:29} into \eqref{eq:27}, we have
\begin{eqnarray}\label{eq:32}
  \begin{split}
   \mathbb E^{\bar u}
   \bigg[\int_0^T
    (l^u(t)-\bar l(t))dt\bigg]={\mathbb E} ^{\bar u}
\bigg[\int_0^T \bigg ( &{ \cal H}(t, x^u(t),
y^u(t), z^u_1(t), z^u_2(t), u(t)) - {\cal H} (t, \bar x(t),
\bar y(t), \bar z_1(t),\bar z_2(t),  \bar u(t))
\\& - \big < {\cal H}_x (t, \bar x(t),
\bar y(t), \bar z_1(t), \bar z_2(t),  \bar u(t)),x^u (t) - \bar
x (t)
\big>
\\&- \big <{ \cal H}_y (t, \bar x(t),
\bar y(t), \bar z_1(t), \bar z_2(t),  \bar u(t)),
y^u (t) - \bar
y (t)\big>
\\&- \big <{\cal H}_{z_1} (t, \bar x(t),
\bar y(t), \bar z_1(t),\bar z_2(t),  \bar u(t)),
 z^u_1 (t) - \bar
z_1 (t)\big>
\\&- \big <{\cal H}_{z_2} (t, \bar x(t),
\bar y(t), \bar z_1(t),\bar z_2(t),  \bar u(t)),
 z^u_2 (t) - \bar
z_2 (t)\big>
\\&-
 \langle (\sigma_2^u(t)-\bar{\sigma}_2(t))
 (h^u(t)-\bar
 h(t)), \bar p(t)\rangle \bigg)dt\bigg] \\
& - {\mathbb E}^u \big [   \left <
x^u (T) - \bar x (T), \bar \Phi_x (T)  \right
> \big ]- {\mathbb E}
 \big [  \left <
y^u (0)
- \bar y (0), \bar \gamma_y (0)  \right
> \big ]
\\& - {\mathbb E}^u
\big [ \la \phi^u (T) - \bar \phi (T), k^u(T)\ra - \left <
\bar \phi_x^*(T)\bar k(T),
x^u (T) - \bar x (T) \right
> \big ]
  \end{split}
  \end{eqnarray}

Then putting\eqref{eq:31} and
 \eqref{eq:32} into \eqref{eq:26} , we get \eqref{eq:21}. The proof is complete.
 \end{proof}

Since the control domain $U$ is convex, for any given admissible controls $u (\cdot) \in {\cal A}$,
the following perturbed control process $u^\epsilon (\cdot)$:
\begin{eqnarray*}
u^\epsilon (\cdot) := \bar u (\cdot) + \epsilon ( u (\cdot) - \bar u (\cdot) ) , \quad  0 \leq \epsilon \leq 1 ,
\end{eqnarray*}
is also in ${\cal A}$. We denote by $(\bar x (\cdot),
\bar y (\cdot),
\bar z_1 (\cdot),\bar z_2(\cdot))$ and
$(x^\epsilon (\cdot), y^\epsilon (\cdot), z^\epsilon_1 (\cdot), z^\eps_2(\cdot))$ the corresponding state processes
associated with $\bar u (\cdot)$ and $u^\epsilon (\cdot)$, respectively. Denote by $(\bar p(\cdot),\bar q_1(\cdot), \bar { q}_2(\cdot),\bar k(\cdot),\bar r(\cdot),\bar R_1(\cdot),\bar R_2(\cdot ) )$ the adjoint process
associated with the admissible pair $(\bar u (\cdot); \bar x (\cdot),
\bar y(\cdot), \bar z_1(\cdot),
 \bar z_2(\cdot))$.

\begin{lemma} \label{lem:3.5}
  Let Assumptions \ref{ass:1.1}
   and \ref{ass:1.2} be satisfied. Then  we have
\begin{eqnarray*}
&&{\mathbb E} \bigg [ \sup_{t\in \cal T} | x^\epsilon (t) - \bar
x (t) |^4  \bigg ]
+{\mathbb E} \bigg [ \sup_{t\in \cal T} | y^\epsilon (t) - \bar
y (t) |^4  \bigg ]
+{\mathbb E} \bigg [
\bigg(\int_0^T| z^\epsilon_1 (t) - \bar
z _1(t) |^2 dt\bigg)^{2} \bigg ]
+{\mathbb E} \bigg [
\bigg(\int_0^T| z^\epsilon_2 (t) - \bar
z _2(t) |^2 dt\bigg)^{2} \bigg ]= O (\epsilon^4) \ .
\end{eqnarray*}
and
\begin{eqnarray*}
&&{\mathbb E} \bigg [ \sup_{t\in \cal T} | \rho^\epsilon (t) - \bar
\rho (t) |^2  \bigg ]
= O (\epsilon^2) \ .
\end{eqnarray*}
\end{lemma}
\begin{proof}
  The proof can be obtained
  directly by Lemmas \ref{lem:3.3}
  and \ref{lem:3.4}.
\end{proof}

Now we are in the position to use  Lemma \ref{lem4} and Lemma \ref{lem:3.5} to derive the variational formula for the cost functional
$J(u(\cdot))$ in terms of the Hamiltonian ${\cal H}$.

\begin{theorem}\label{them:3.1}
 Let Assumptions \ref{ass:1.1} and
 \ref{ass:1.2} be
 satisfied.
Then for any admissible control $u (\cdot) \in {\cal A}$, the directional derivative of the cost functional $J (u (\cdot))$ at $\bar u (\cdot)$  in the direction $u (\cdot) - \bar u(\cdot)$
exists and is given by
\begin{eqnarray}\label{eq:4.4}
&& \frac{d}{d\epsilon} J ( \bar u (\cdot) + \epsilon ( u (\cdot) - \bar u (\cdot) ) ) |_{\epsilon=0} \nonumber \\
&& := \lim_{\epsilon \rightarrow 0^+}
\frac{ J ( \bar u (\cdot) + \epsilon ( u (\cdot) - \bar u (\cdot) ) )
- J( \bar u (\cdot) ) }{\epsilon} \nonumber \\
&& = {\mathbb E}^{\bar u} \bigg [ \int_0^T \left < {\cal H}_u (t, \bar x(t),
\bar y(t), \bar z_1(t),\bar z_2(t),  \bar u(t)),
u (t) -
 \bar u (t) \right > d t \bigg ] .
\end{eqnarray}
\end{theorem}

\begin{proof}
For notational simplicity, write
\begin{eqnarray}
\beta^\epsilon &:=& {\mathbb E} ^{\bar u}
\bigg[\int_0^T \bigg [ { \cal H}(t, x^{u^\eps}(t),
y^{u^\eps}(t), z_1^{u^\eps}(t),z_2^{u^\eps}(t), u(t)) - {\cal H} (t, \bar x(t),
\bar y(t), \bar z_1(t),\bar z_2(t),  \bar u(t)) \nonumber \\&& - \big < {\cal H}_x (t, \bar x(t),
\bar y(t), \bar z_1(t),\bar z_2(t),  \bar u(t)),x^{u^\eps} (t) - \bar
x (t)
\big>- \big <{ \cal H}_y (t, \bar x(t),
\bar y(t), \bar z_1(t),
\bar z_2(t),  \bar u(t)),
y^{u^\eps} (t) - \bar
y (t)\big>\nonumber
\\&&- \big <{\cal H}_{z_1} (t, \bar x(t),
\bar y(t), \bar z_1(t),
\bar z_2(t), \bar u(t)),
 z_1^{u^\eps} (t) - \bar
z_1 (t)\big>- \big <{\cal H}_{z_2} (t, \bar x(t),
\bar y(t), \bar z_1(t),\bar z_2(t),  \bar u(t)),
 z_2^{u^\eps} (t) - \bar
z_2(t)\big>\nonumber
\\&& - \big <{\cal H}_{u} (t, \bar x(t),
\bar y(t), \bar z_1(t),\bar z_2(t),  \bar u(t)),
 u^\eps (t) - \bar
u(t)\big>-
 \langle (\sigma_2^{u^\eps}(t)-\bar{\sigma}_2(t))
 (h^{u^\eps}(t)-\bar
 h(t)), \bar p(t)\rangle \nonumber
 \\&& -
 \langle (z_2^{u^\eps}(t)-\bar{z}_2(t))
 (h^{u^\eps}(t)-\bar
 h(t)), \bar k(t)\rangle dt\bigg]
\nonumber \\
&& + {\mathbb E}^{\bar u} \big [ \Phi^{u^\eps}(T)
 - \bar \Phi (T) - \left <
x^{u^\eps} (T) - \bar x (T), \bar \Phi_x (T)  \right
> \big ]
\nonumber \\
&& - {\mathbb E}^{\bar u}
\big [ \la \phi^{u^\eps} (T) - \bar \phi (T), \bar k(T)\ra - \left <
\bar \phi_x^*(T) \bar k(T),
x^{u^\eps} (T) - \bar x (T) \right
> \big ]
\nonumber \\
&& + {\mathbb E}
 \big [ \gamma^{u^\eps}(0) - \bar \gamma (0) - \left <
y^{u^\eps} (0)
- \bar y (0), \bar \gamma_y (0)  \right
> \big ]
\nonumber\nonumber
\\&&+\mathbb E\bigg[\int_0^T \bar R_2(t)(\rho^{u^\eps}(t)-\rho^{\bar u}(t))(h^{u^\eps}(t)-\bar
 h(t))dt\bigg] \nonumber
\\&&+\mathbb E\bigg[\int_0^T(l^{u^\eps}(t)-\bar l(t))( \rho^{u^\eps}(t)-\bar
\rho(t))dt\bigg]\nonumber
\\&&
+\mathbb E \bigg[ ( \rho^{u^\eps} (T) - \bar \rho(T)) (\Phi^{u^\eps} (T) -
\bar \Phi (T))\bigg]
\end{eqnarray}
By Lemma \ref{lem4}, we have
\begin{eqnarray}\label{eq:4.13}
J ( u^\epsilon (\cdot) ) - J ( \bar u (\cdot) )
= \beta^\epsilon + \epsilon {\mathbb E}^
{\bar u} \bigg [ \int_0^T \left < {\cal H}_u(t, \bar x(t),
\bar y(t), \bar z_1(t),\bar z_2(t),  \bar u(t)), u (t) - \bar u (t) \right > d t \bigg ] .
\end{eqnarray}
Under Assumptions \ref{ass:1.1} and
\ref{ass:1.2}, combining the Taylor Expansions, Lemma \ref{lem:3.5}, and the dominated
convergence theorem, we have
\begin{eqnarray}\label{eq:4.121}
\beta^\epsilon = o(\epsilon) .
\end{eqnarray}
Plugging \eqref{eq:4.121} into \eqref{eq:4.13} gives
\begin{eqnarray*}
\lim_{\epsilon \rightarrow 0^+} \frac{J (u^\epsilon (\cdot) ) - J (\bar u (\cdot))}{\epsilon}
= {\mathbb E}^{\bar u} \bigg [ \int_0^T \left < {\cal H}_u (t, \bar x(t),
\bar y(t), \bar z_1(t),\bar z_2(t),  \bar u(t)),
u (t) - \bar u (t) \right > d t \bigg ] .
\end{eqnarray*}
This completes the proof.
\end{proof}

Now we derive the necessary condition and sufficient maximum principles for Problem \ref{pro:1.1} or \ref{pro:1.2}.
We first give the necessary condition of optimality for the existence of an optimal control.

\begin{theorem}[{\bf Necessary Stochastic Maximum principle}]
 Let Assumptions \ref{ass:1.1}
  and \ref{ass:1.2} be satisfied. Let $( \bar u (\cdot); \bar x (\cdot), \bar y(\cdot), \bar z_1(\cdot),
  \\\bar z_2(\cdot),\bar\rho(\cdot) )$ be
an optimal pair of  Problem \ref{pro:1.2}. Then
\begin{eqnarray}\label{eq:4.15}
\left <  {\mathbb E}^{\bar u} [ {\cal H}_{u} (t, \bar x(t),
\bar y(t), \bar z_1(t),\bar z_2(t),  \bar u(t))  | {\mathscr F}_t ^Y ],
v - \bar u (t) \right > \geq 0 , \quad \forall v \in U ,\ a.e.\ a.s..
\end{eqnarray}
\end{theorem}

\begin{proof}
Since all admissible controls
are $\{\mathscr F^Y_t\}_{t\in \cal T}$-adapted
processes, from the property of conditional
expectation, Theorem \ref{them:3.1} and the optimality of $\bar u(\cdot)$,  we
deduce that
\begin{eqnarray*}
&& {\mathbb E} \bigg [ \int_0^T
 \langle {\mathbb E}
  [{\bar \rho(t)}{ \bar {\cal H}}_u (t, \bar x(t),
\bar y(t), \bar z_1(t),\bar z_2(t),  \bar u(t)) | {\mathscr F}_t ^Y] ,
u (t) - \bar u (t) \rangle  d t \bigg ] \\
&& = {\mathbb E} \bigg[\int_0^T \langle
{\bar \rho(t)}{ \bar  {\cal H}}_u (t, \bar x(t),
\bar y(t), \bar z_1(t),\bar z_2(t),  \bar u(t))
, u (t) - \bar u (t) \rangle \d t \bigg ]
\\
&& = {\mathbb E}^{\bar u} \bigg[\int_0^T \langle{ \bar {\cal H}}_u (t, \bar x(t),
\bar y(t), \bar z_1(t),\bar z_2(t),  \bar u(t))
, u (t) - \bar u (t) \rangle \d t \bigg ] \\
&& = \lim_{\epsilon \rightarrow 0^+} \frac{J( \bar u (\cdot) + \epsilon (
u (\cdot) - \bar u (\cdot) ) ) - J ( \bar u (\cdot) )}{\epsilon} \geq 0 ,
\end{eqnarray*}
which implies that
\begin{eqnarray}
  \langle {\mathbb E}
  [{\bar \rho(t)}{ \bar {\cal H}}_u ( t, \bar x(t),
\bar y(t), \bar z_1(t),\bar z_2(t),  \bar u(t)) | {\mathscr F}_t ^Y] ,
v - \bar u (t) \rangle \geq 0, \quad \forall v \in U ,\ a.e.\ a.s..
\end{eqnarray}
On the other hand, since
$\bar \rho(t)>  0, $
 \begin{eqnarray}
   \begin{split}
    & \left <  {\mathbb E}^{\bar u} [ {\cal H}_{u} (t, \bar x(t),
\bar y(t), \bar z_1(t),\bar z_2(t),  \bar u(t))  | {\mathscr F}_t ^Y ],
v - \bar u (t) \right >
\\&=
\frac{1}{\mathbb E[\bar\rho (t)|{\mathscr F}_t ^Y]}
\big\langle {\mathbb E}
  [{\bar \rho(t)}{ \bar {\cal H}}_u (t, \bar x(t),
\bar y(t), \bar z_1(t),\bar z_2(t),  \bar u(t)) | {\mathscr F}_t ^Y] ,
v - \bar u (t)\big \rangle
\\& \geq 0.
   \end{split}
 \end{eqnarray}
  The proof is complete.

\end{proof}

Next we give the sufficient condition of optimality for the existence of an optimal control of
Problem \ref{pro:1.2} in the case when
the observation process is not
affected by the control process.
Suppose that
$$h(t,x,u)=h(t)$$ is an
$\mathscr F^Y_t-$ adapted bounded
process.
  Define a new probability measure $\mathbb Q$ on $(\Omega, \mathscr F)$ by
\begin{eqnarray}
  d\mathbb Q=\rho(1)d\mathbb P,
\end{eqnarray}
where
\begin{eqnarray} \label{eq:43}
  \left\{
\begin{aligned}
  d \rho(t)=&  \rho(t) h(s)dY(s)\\
  \rho(0)=&1.
\end{aligned}
\right.
\end{eqnarray}

\begin{theorem}{\bf [Sufficient Maximum Principle] } \label{thm:4.1}
 Let Assumptions \ref{ass:1.1}
  and \ref{ass:1.2} be satisfied. Let $(\bar u (\cdot);
  \bar x (\cdot),
  \bar y(\cdot), \bar z_1(\cdot),\bar z_2 (\cdot))$ be an admissible pair with $\phi(x)=\phi x,$
  where $\phi$ is $\mathscr F_T-$measurable bounded  random variable.
If the following conditions are satisfied,
\begin{enumerate}
\item[(i)]  $\Phi$ and $\gamma$ is convex in $x$ and $y,$ respectively,
\item[(ii)] the Hamiltonian ${\cal H}$ is convex in $(x, y, z_1, z_2,  u)$,
\item[(iii)]
\begin{eqnarray*}\label{eq:5.119}
&& \mathbb E\bigg[{\cal H} ( t,
\bar x (t),
\bar y (t), \bar z_1 (t),
\bar z_2 (t),\bar u(t)) |\mathscr F^Y_t\bigg]\nonumber \\
&& = \min_{u \in U }
 \mathbb E\bigg[{\cal H} ( t,
\bar x (t),
\bar y (t), \bar z_1 (t),\bar z_2 (t), u)
 |\mathscr F^Y_t\bigg], \quad \mbox {a.e.\ a.s.} ,
\end{eqnarray*}
\end{enumerate}
then $(\bar u (\cdot),  \bar x (\cdot),
\bar y(\cdot), \bar z_1(\cdot), \bar z_2 (\cdot))$ is an optimal pair of Problem \ref{pro:1.2}.
\end{theorem}

\begin{proof}
Let $({u}(\cdot);x^u(\cdot),
 y^u(\cdot),
 z^u_1(\cdot),  z^u_2(\cdot),
\rho^u(\cdot))$  be an arbitrary admissible pair. By Lemma \ref{lem4},
we can represent the difference $J ( u (\cdot) ) - J ( \bar u (\cdot) )$ as follows
\begin{eqnarray}\label{eq:40}
&&J (u (\cdot)) - J (\bar u(\cdot))\nonumber
\\ &=& {\mathbb E} ^Q
\bigg[\int_0^T \bigg [ { \cal H}(t, x^u(t),
y^u(t), z^u(t), u(t)) - {\cal H} (t, \bar x(t),
\bar y(t), \bar z(t),  \bar u(t)) \nonumber \\&& - \big < {\cal H}_x (t, \bar x(t),
\bar y(t), \bar z(t),  \bar u(t)),x^u (t) - \bar
x (t)
\big>- \big <{ \cal H}_y (t, \bar x(t),
\bar y(t), \bar z(t),  \bar u(t)),
y^u (t) - \bar
y (t)\big>\nonumber
\\&&- \big <{\cal H}_z (t, \bar x(t),
\bar y(t), \bar z(t),  \bar u(t)),
 z^u (t) - \bar
z (t)\big>
 dt\bigg]
\nonumber \\
&& + {\mathbb E}^{\bar u} \big [ \Phi(T)
 - \bar \Phi^u (T) - \left <
x^u (T) - \bar x (T), \bar \Phi_x (T)  \right
> \big ]
\nonumber \\
&& + {\mathbb E}
 \big [ \gamma^u(0) - \bar \gamma (0) - \left <
y^u (0)
- \bar y (0), \bar \gamma_y (0)  \right
> \big ].
\end{eqnarray}

By the convexity of ${\cal H}$, $\Phi$
and $\gamma$ (i.e. Conditions (i) and (ii)), we have
\begin{eqnarray}\label{eq:41}
 &&{ \cal H}(t, x^u(t),
y^u(t), z^u_1(t), z^u_2(t), u(t)) - {\cal H} (t, \bar x(t),
\bar y(t), \bar z_1(t),\bar z_2(t),  \bar u(t))\nonumber
\\ &\geq&  \big < {\cal H}_x (t, \bar x(t),
\bar y(t), \bar z_1(t), \bar z_2(t),  \bar u(t)),x^u (t) - \bar
x (t)
\big>\nonumber
+ \big <{ \cal H}_y (t, \bar x(t),
\bar y(t), \bar z_1(t), \bar z_2(t),  \bar u(t)),
y^u (t) - \bar
y (t)\big>\nonumber
\\&&+ \big <{\cal H}_{z_1} (t, \bar x(t),
\bar y(t), \bar z_1(t),\bar z_2(t),  \bar u(t)),
 z^u_1 (t) - \bar
z_1 (t)\big>\nonumber
+ \big <{\cal H}_{z_2} (t, \bar x(t),
\bar y(t), \bar z_1(t),\bar z_2(t),  \bar u(t)),
 z^u_2 (t) - \bar
z_2 (t)\big>
\\&&+ \big <{\cal H}_{u} (t, \bar x(t),
\bar y(t), \bar z_1(t),\bar z_2(t),  \bar u(t)),u (t) - \bar
u(t)\big>,
\end{eqnarray}

\begin{eqnarray}\label{eq:42}
 \Phi^u(T) - \bar \Phi (T) \geq \left < X^u (T) - \bar X (T), \Phi_x (T) \right >
\end{eqnarray}
and
\begin{eqnarray}\label{eq:43}
 \gamma^u(0) - \bar \gamma (0) \geq \left < y^u (0) - \bar y (0), \gamma_y (0) \right >.
\end{eqnarray}

Furthermore, from the optimality condition (iii)
and the convex optimization principle (see Proposition 2.21 of \cite{ET1976}), we have
\begin{eqnarray}\label{eq:5.5}
\left < u (t) - \bar u (t), \mathbb E\bigg[{\cal H}_{u} ( t,
\bar x (t),
\bar y (t), \bar z_1 (t),
\bar z_2(t),\bar u(t)) |\mathscr F^Y_t\bigg] \right >
\geq 0 .
\end{eqnarray}
which imply that

\begin{eqnarray}\label{eq:45}
\mathbb E^Q\bigg[\left < u (t) - \bar u (t), {\cal H}_{u} ( t,
\bar x (t),
\bar y (t), \bar z_1(t),\bar z_2(\cdot),\bar u(t)) \right >
\bigg]\geq 0.
\end{eqnarray}

Putting \eqref{eq:41},\eqref{eq:42},\eqref{eq:43} and  \eqref{eq:45} into \eqref{eq:40},
we have
\begin{eqnarray}
J (u (\cdot)) - J (\bar u (\cdot)) \geq 0 .
\end{eqnarray}
Due to the arbitrariness of $u (\cdot)$, we can conclude that $\bar u (\cdot)$ is an optimal control process
and thus $(\bar u (\cdot), \bar x (\cdot), \bar y(\cdot), \bar z_1(\cdot), z_2(\cdot))$ is an optimal pair. The proof is completed.
\end{proof}

\bibliographystyle{model1a-num-names}

\end{document}